%% file: main.tex
\documentclass{amsart}
\usepackage[utf8]{inputenc}
\usepackage[T1]{fontenc}
\usepackage[english]{babel}

\usepackage{latexsym}
\usepackage{amssymb}
\usepackage{amsthm}
\usepackage{amsmath}
\usepackage{csquotes}
\usepackage{enumitem}

\newcommand{\R}{{\mathbb{R}}}
\newcommand{\N}{{\mathbb{N}}}
\DeclareMathOperator{\ext}{ext}
\DeclareMathOperator{\Img}{Im}
\newcommand{\Cl}[1]{\overline{#1}}

\newtheorem{theorem}{Theorem}[section]
\newtheorem{proposition}[theorem]{Proposition}

\title{Two new examples of Banach spaces with a plastic unit ball}
\author{Rainis Haller}
\address{Institute of Mathematics and Statistics, University of Tartu, Narva
mnt 18, 51009 Tartu, Estonia}
\email{rainis.haller@ut.ee}
\author{Nikita Leo}
\address{Institute of Mathematics and Statistics, University of Tartu, Narva
mnt 18, 51009 Tartu, Estonia}
\email{nikita.leo@ut.ee}
\author{Olesia Zavarzina}
\address{Department of Mathematics and Informatics, V. N. Karazin Kharkiv National
University, 61022 Kharkiv, Ukraine.}
\email{olesia.zavarzina@yahoo.com}

\date{}

\usepackage[maxbibnames=5, style = numeric, backend = biber]{biblatex}
\begin{document}

\begin{abstract}
    We prove that Banach spaces $\ell_1\oplus_2\mathbb{R}$ and $X\oplus_\infty Y$, with strictly convex $X$ and $Y$, have plastic unit balls (we call a metric space plastic if every non-expansive bijection from this space onto itself is an isometry).
\end{abstract}
\maketitle
\input{1_introduction}
\input{2_preliminaries}
\newpage
\input{3_olesia}
\newpage
\input{4_nikita}
\newpage
\section*{Acknowledgements}
 The third author is grateful to her scientific advisor Vladimir Kadets for help with this project. The third author was supported  by the  National Research Foundation of Ukraine funded by Ukrainian State budget in frames of project 2020.02/0096 ``Operators in infinite-dimensional spaces:  the interplay between geometry, algebra and topology''.
\printbibliography[title={References}]

\end{document}

%% file: 1_introduction.tex
\section{Introduction}
A function between two metric spaces is called an \emph{isometry} if it preserves distances between points, and \emph{non-expansive} if it does not increase distances between points. We call a metric space $X$ \emph{plastic} if every non-expansive bijection from $X$ onto itself is an isometry. The last notion was introduced by S.~A.~Naimpally, Z.~Piotrowski, and E.~J.~Wingler in \cite{NPW2006}. It is known that every totally bounded metric space is plastic, see \cite[Satz IV]{FH1936} or \cite[Theorem 1.1]{NPW2006}. On the other hand, a plastic metric space need not be totally bounded nor bounded -- e.g., the set of integers with the usual metric is plastic \cite[Theorem 3.1]{NPW2006}. There are also examples of bounded metric spaces that are not plastic, one of our favorite examples here is a solid ellipsoid in Hilbert space $\ell_2(\mathbb Z)$ with infinitely many semi-axes equal to $1$ and infinitely many semi-axes equal to $2$, see \cite[Example 2.7]{CKOW2016}.

It is a challenging open problem, posed by B.~Cascales, V.~Kadets, J.~Orihuela, and E.~J.~Wingler in 2016 \cite{CKOW2016}, whether the unit ball of every Banach space is a plastic metric space. The unit ball of a finite-dimensional space is compact, and therefore plastic. So the question is really just about the infinite-dimensional spaces. So far, the plasticity of the unit ball has successfully been proved for the following spaces and classes of spaces:
\begin{itemize}
    \item strictly convex spaces,
    \item the space $\ell_1$,
    \item $\ell_1$-sums of strictly convex spaces,
    \item spaces whose unit sphere is the union of all its finite-dimensional polyhedral extreme subsets,
    \item the space $c$.
\end{itemize}
The first result was for strictly convex spaces and it was obtained by B.~Cascales, V.~Kadets, J.~Orihuela, and E.~J.~Wingler in 2016 \cite{CKOW2016}. The same year, V.~Kadets and O.~Zavarzina proved the plasticity of the unit ball of $\ell_1$ \cite{KZ2016}. They generalized this result to $\ell_1$-sums of strictly convex spaces in 2017 \cite{KZ2018}. The fourth item was obtained by C.~Angosto, V.~Kadets, and O.~Zavarzina in 2018 \cite[Theorem 4.11]{AKZ2018}. The plasticity of the unit ball of $c$ was proved by N.~Leo in 2021 \cite[Theorem 4.1]{Leo2021}.

In this paper, we present two new examples of Banach spaces whose unit balls are plastic: the sum of $\ell_1$ and $\R$ by $\ell_2$, and the sum of any two strictly convex spaces by $\ell_\infty$.

%% file: 2_preliminaries.tex
\section{Preliminaries}
We consider only real Banach paces. For a Banach space $X$, we denote the closed unit ball and the unit sphere of $X$ by $B_X$ and $S_X$.

Extreme points turn out to be essential to the study of plasticity of the unit ball. Recall that for a vector space $X$ and for a convex subset $C$ of $X$, a point $x\in C$ is called an \emph{extreme point} of $C$ if it does not belong to the interior of any non-trivial line segment connecting two distinct points of $C$. We use $\ext C$ to denote the set of extreme points of a set $C$. Henceforth, we focus on extreme points of the unit ball. Extreme points of the unit ball lie on the unit sphere. A space such that all the points of the unit sphere are extreme is called \emph{strictly convex}. Strictly convex spaces have a property that, for any two distinct points $x$ and $y$ and any non-negative scalars $\alpha$ and $\beta$ with $\alpha+\beta=\|x-y\|$, the point $\frac{\beta}{\|x-y\|}x+\frac{\alpha}{\|x-y\|}y$ is the only point of the space that is distance $\alpha$ from $x$ and distance $\beta$ from $y$. This is going to be used in the proof for $\ell_\infty$-sum of strictly convex spaces.

The next proposition describes the behaviour of a non-expansive bijection from the unit ball of a Banach space onto itself. It provides some of the main tools for the study of plasticity of the unit ball. The last item is the reason why are extreme points so important to the topic at hand. This proposition is going to be used extensively throughout both of the proofs.

\begin{proposition}[{\cite[Theorem 2.3]{CKOW2016}}]\label{proposition_BnEproperties}
    Let $X$ be a Banach space and let $F\colon B_X\to B_X$ be a non-expansive bijection. Then
        \begin{enumerate}
            \item $F(0)=0$,
            \item for each $x\in B_X$, $\|F(x)\|\leq\|x\|$,
            \item if $x\in S_X$, then $F^{-1}(x)\in S_X$,
            \item if $x\in\ext B_X$, then $F^{-1}(x)\in \ext B_X$ and $F^{-1}(\alpha x)=\alpha F^{-1}(x)$ for each $\alpha\in[-1,1]$.
        \end{enumerate}
\end{proposition}

It is also useful to know the following result.
\begin{theorem}[{\cite[Theorem 2]{Mankiewicz1972}}]\label{theorem_Mankiewicz}
    Let $X$ and $Y$ be normed spaces and let $U$ be a subset of $X$ and $V$ be a subset of $Y$. If $U$ and $V$ are convex with non-empty interior and there exists an isometric bijection $F\colon U\to V$, then $F$ extends to an affine isometric bijection $\widetilde{F}\colon X\to Y$.
\end{theorem}
The theorem implies that if $X$ is a Banach space and $F\colon B_X\to B_X$ is an isometric bijection, then $F$ extends to an isometric automorphism of $X$ -- a linear isometric bijection from $X$ onto itself. If we want to prove that every non-expansive bijection on the unit ball of some space $X$ is an isometry, it might be useful to know what isometric bijections are there. The last result says that these are precisely the restrictions of the isometric automorphisms of $X$.

There is also a sufficient condition for a non-expansive bijection $F\colon B_X\to B_X$ to be an isometry.
\begin{theorem}[{\cite[Theorem 2.6]{CKOW2016}}]\label{theorem_sufficient}
    Let $X$ be a Banach space and let $F\colon B_X\to B_X$ be a non-expansive bijection. If $F(S_X)=S_X$ and $F(\alpha x)=\alpha F(x)$ for every $x\in S_X$ and every $\alpha\in[-1,1]$, then $F$ is an isometry.
\end{theorem}
The latter will be used to finish the proof for the case of $\ell_\infty$-sum of two strictly convex spaces.

%% file: 3_olesia.tex
\section{The space $\ell_1\oplus_2\mathbb{R}$}

\begin{theorem}
The unit ball of $\ell_1\oplus_2\mathbb{R}$ is a plastic metric space.
\end{theorem}

\begin{proof}
Denote the space $\ell_1\oplus_2\mathbb{R}$ by $Z$. Let $e_k$ stand for the $k$-th element of the canonical basis of $\ell_1$. Denote by $h$ the projection of $Z$ onto $\R$. For each $b\in[-1,1]$ denote by $L_b$ the set $\{z\in B_Z\colon h(z)=b\}$. Note that
\[\ext B_Z=\{(ae_m, b)\colon a,b\in\R,\ a^2+b^2=1,\ m\in\N\}.\]

Let $(ae_i,b)$ and $(ce_j,d)$ be two arbitrary extreme points. If $i=j$, then the distance between these two points is $\sqrt{(a-c)^2+(b-d)^2}$ and we have
\begin{align*}
    \sqrt{(a-c)^2+(b-d)^2}=2 &\iff(a-c)^2+(b-d)^2=4\\
    &\iff a^2-2ac+c^2+b^2-2bd+d^2=4\\
    &\iff 0=a^2+2ac+c^2+b^2+2bd+d^2\\
    &\iff 0=(a+c)^2+(b+d)^2\\
    &\iff a=-c\ \&\ b=-d.
\end{align*}
If $i\ne j$, then the distance between $(ae_i,b)$ and $(ce_j,d)$ is $\sqrt{(|a|+|c|)^2+(b-d)^2}$ and we have
\begin{align*}
    \sqrt{(|a|+|c|)^2+(b-d)^2}=2 &\iff(|a|+|c|)^2+(b-d)^2=4\\
    &\iff a^2+2|a||c|+c^2+b^2-2bd+d^2=4\\
    &\iff 0=a^2-2|a||c|+c^2+b^2+2bd+d^2\\
    &\iff 0=(|a|-|c|)^2+(b+d)^2\\
    &\iff |a|=|c|\ \&\ b=-d.
\end{align*}
That is, the distance between points $(ae_i,b)$ and $(ce_j,d)$ is equal to two if and only if either $i=j$, $a=-c$, and $b=-d$, or $i\ne j$, $|a|=|c|$, and $b=-d$. In particular, if $u$ and $v$ are extreme points that are distance two apart, then $h(u)=-h(v)$. This is going to be used in Step 1.

{ \bf Step 1.} Let $F$ be a non-expansive bijection from $B_Z$ onto itself. We need to show that $F$ is an isometry. Let us first show that for each $m\in\N$ we have $F^{-1}(e_m,0)=(\theta e_n, 0)$, where $\theta\in\{-1,1\}$ and $n\in\N$.

Let $m\in\N$ be arbitrary. Choose indices $i$ and $j$ such that $m$, $i$, and $j$ are pairwise distinct. Note that $x=(e_m,0)$, $y=(e_i,0)$, and $z=(e_j,0)$ are extreme points. Item (4) of Proposition \ref{proposition_BnEproperties} implies that $x'=F^{-1}(x)$, $y'=F^{-1}(y)$, and $z'=F^{-1}(z)$ are also extreme points. Note that $x$, $y$, and $z$ are distance two apart. The non-expansiveness of $F$ implies that $x'$, $y'$, and $z'$ are also distance two apart. Since $x'$, $y'$, and $z'$ are extreme points that are distance two apart, we have $h(x')=-h(y')$, $h(y')=-h(z')$, and $h(z')=-h(x')$, from which it follows that $h(x')=0$. As $x'$ is an extreme point, it has the form $(ae_n,b)$, where $a,b\in\R$, $a^2+b^2=1$, and $n\in\N$. Since $h(x')=0$, we have $b=0$ and $a\in\{-1,1\}$, which finishes the proof. For further reference define $\sigma_m=n$, $\theta_m=a$, and $g_m=\theta_m e_{\sigma(m)}$.

This way we obtain a function $\sigma:\N\to\N$, which describes the permutation of indices exerted by $F$. Note that this function is an injection. Indeed, let $i$ and $j$ be two indices with $\sigma(i)=\sigma(j)$. Denote by $n$ the common value of $\sigma(i)$ and $\sigma(j)$. Then $F^{-1}(e_i,0)=(\theta_i e_n,0)$ and $F^{-1}(e_j,0)=(\theta_j e_n,0)$, where $\theta_i,\theta_j\in\{-1,1\}$. Suppose that $\theta_j=-\theta_i$. Then item (4) of Proposition \ref{proposition_BnEproperties} implies $F^{-1}(-e_i,0)=(\theta_j e_n,0)$, but the latter implies $(-e_i,0)=(e_j,0)$, which is impossible. Hence $\theta_i=\theta_j$, so we get $(e_i,0)=(e_j,0)$, from which it follows that $i=j$. Therefore, $\sigma$ is indeed an injection. This fact is going to be implicitly used in what follows. However, note that we do not yet know whether $\sigma$ is also a surjection.

Now we know that $F^{-1}(e_m,0)=(g_m,0)$ for every $m\in\N$. Further, item (4) of Proposition \ref{proposition_BnEproperties} implies $F^{-1}(ae_m,0)=(ag_m,0)$ for every $a\in[-1,1]$. The latter can be restated as that $F(ag_n,0)=(ae_n,0)$ for every $n\in\N$ and $a\in[-1,1]$.

{ \bf Step 2.} Using the same argument as in \cite{KZ2016}, we can show that for every $N \in \N$ and real numbers $a_1,\ldots,a_N$ with $\sum_{n=1}^N|a_n| \leq 1$ we have
\[F\left(\sum_{n=1}^Na_n g_n,0\right) = \left(\sum_{n=1}^Na_n e_n,0\right),\]
where $g_n$ is defined as in Step 1. The proof is by induction on $N$. We are going to omit the proof, since it repeats the argument from \cite{KZ2016} almost word to word. The proof depends on the fact that $F(ag_n,0)=(ae_n,0)$ for every $n\in\N$ and $a\in[-1,1]$ as was established in Step 1.

{ \bf Step 3.} The previous step and continuity of $F$ imply that for every real sequence $(a_n)_{n=1}^{\infty}$ with $\sum_{n=1}^{\infty}|a_n|\leq 1$ we have
\[F\left(\sum_{n=1}^\infty a_n g_n,0\right) = \left(\sum_{n =1}^\infty a_n e_n,0\right).\]
If we denote the set $\{z\in B_Z\colon h(z)=0\}$ by $L_0$, the latter implies $F^{-1}(L_0)\subset L_0$.

{ \bf Step 4.} We have seen that $F^{-1}(L_0)\subset L_0$. Now, let us show that $F^{-1}(L_0)=L_0$. We are going to use a proof by contradiction. Suppose that $F^{-1}(L_0)\ne L_0$. 
Then there is $z\in L_0 \setminus F^{-1}(L_0)$. Let $s$ stand for $F^{-1}(0,1)$. Item (4) of Proposition \ref{proposition_BnEproperties} implies that $F^{-1}(0,-1)=-F^{-1}(0,1)$. Hence $F(s)=(0,1)$ and $F(-s)=(0,-1)$. Let us consider two cases.

{ \it Case 1.} Suppose that $s\not\in L_0$. Consider a continuous curve $\xi:[0,1]\to B_Z$ composed of line segments $[s,z]$ and $[z,-s]$. Note that $\Img\xi\cap L_0=\{z\}$, but $z\not\in F^{-1}(L_0)$. Hence $\Img\xi$ does not intersect $F^{-1}(L_0)$. Since $F$ is continuous, $\xi'=F\circ\xi$ is also a continuous curve in $B_Z$. We know that $\Img\xi$ does not intersect $F^{-1}(L_0)$, so $\Img\xi'$ should not intersect $L_0$. Since $h$ is a continuous functional, $f=h\circ\xi'$ is a continuous function from $[0,1]$ into $\R$. Note that $f(0)=1$ and $f(1)=-1$. Therefore, there exists $t\in(0,1)$ such that $f(t)=0$, which implies that $\Img\xi'\cap L_0\ne\emptyset$, which is a contradiction.

{ \it Case 2.} For the second case, suppose that $s\in L_0$. The argument is the same except that now we consider a curve composed of line segments $[s,(0,1)]$ and $[(0,1),-s]$. Again, we see that $\Img\xi$ does not intersect $F^{-1}(L_0)$, thus $\Img\xi'$ should not intersect $L_0$. However, we have $f(0)=1$ and $f(1)=-1$, which implies that there exists $t\in(0,1)$ with $f(t)=0$. The latter means that $\Img\xi'\cap L_0\ne\emptyset$, which is again a contradiction.

Therefore, we have $F^{-1}(L_0)=L_0$. Note that this implies that the function $\sigma$ defined earlier is a bijection from $\N$ onto $\N$. This allows us to give the set of extreme points an alternative description:
\[\ext B_Z=\{(ag_n, b)\colon a,b\in\R,\ a^2+b^2=1,\ n\in\N\}.\]

{ \bf Step 5.} Let us show that $F(0,1)=(0,1)$ or $F(0,1)=(0,-1)$. First, item (4) of Proposition \ref{proposition_BnEproperties} implies that $F^{-1}(0,1)\in \ext B_Z$. Hence we have $F^{-1}(0,1)=(ag_n,b)$, where $a,b\in\R$, $a^2+b^2=1$, and $n\in\N$. Consider points $(ag_n,b)$ and $(ag_n,0)$. These points are distance $|b|$ from each other. We have $F(ag_n,b)=(0,1)$ and by Step 1 we have $F(ag_n,0)=(ae_n,0)$. The distance between these points is $\sqrt{1+a^2}$. Since $F$ is non-expansive, we must have $\sqrt{1+a^2}\leq |b|$, which implies that $1+a^2\leq b^2$. Keeping in mind the relation $a^2+b^2=1$, we can infer that $a=0$ and $b\in\{-1,1\}$. Therefore, either $F(0,1)=(0,1)$ or $F(0,1)=(0,-1)$.

For the case $F(0,1)=(0,1)$, item (4) of Proposition \ref{proposition_BnEproperties} implies that $F(0,b)=(0,b)$ for all $b\in [-1,1]$. In particular, we have $F(0,-1)=F(0,-1)$. For the case $F(0,1)=(0,-1)$, item (4) of Proposition \ref{proposition_BnEproperties} implies that $F(0,b)=(0,-b)$ for all $b\in [-1,1]$. In particular, we have $F(0,-1)=F(0,1)$. Further we deal with the case $F(0,1)=(0,1)$. The case $F(0,1)=(0,-1)$ can be handled in an analogous way.

{ \bf Step 6.} Note that an operator $\mathcal{L}\colon\ell_1\to\ell_1$ defined by the formula \[\mathcal{L}\left(\sum_{n=1}^\infty a_ng_n\right)=\sum_{n=1}^\infty a_ne_n\]
is an isometric automorphism of $\ell_1$. Hence the operator $\mathcal{L}'\colon Z\to Z$ defined by $\mathcal{L}'(x,y)=(\mathcal{L}(x),y)$ is an isometric automorphism of $Z$. Our goal is to show that $F$ is the restriction of the latter operator to $B_Z$. Obviously, this will also prove that $F$ is an isometry. So we need to show that $F(z)=\mathcal{L}'(z)$ for every $z\in B_Z$. Note that the latter has already been established for $h(z)=0$ and $h(z)\in\{-1,1\}$. Hence, we only need to consider the case $|h(z)|\in(0,1)$. Let $b$ be an arbitrary real number with $|b|\in(0,1)$, let $a$ stand for $\sqrt{1-b^2}$, and let us show that $F(z)=\mathcal{L}'(z)$ for $h(z)=b$. Note that $z$ has the form $(ax,b)$, where $x\in B_{\ell_1}$. Also note that $\mathcal{L}'(z)=(\mathcal{L}(ax),b)$. Therefore, we need to show that $F(ax,b)=(\mathcal{L}(ax),b)$ for every $x\in B_{\ell_1}$.

We begin with showing that $F(ax,b)=(\mathcal{L}(ax),b)$ holds for every $x\in S_{\ell_1}$. This is the same as to show that $F^{-1}(ay,b)=(\mathcal{L}^{-1}(ay),b)$ for every $y\in S_{\ell_1}$. Hence, fix an arbitrary $y\in S_{\ell_1}$. Since $F^{-1}(ay,b)\in B_Z$, it has the form $(cx,d)$, where $c,d\in\R$, $c^2+d^2=1$, and $x\in B_{\ell_1}$. Further, since $(ay,b)\in S_Z$, by item (3) of Proposition \ref{proposition_BnEproperties} we have $(cx,d)\in S_Z$. Therefore, $x\in S_{\ell_1}$.

Consider points $(cx,d)$ and $(0,1)$. The distance between these points is equal to $\sqrt{c^2+(1-d)^2}=\sqrt{2-2d}$. We have $F(cx,d)=(ay,b)$ and by Step 5 we have $F(0,1)=(0,1)$. The distance between these points is $\sqrt{a^2+(1-b)^2}=\sqrt{2-2b}$. Since $F$ is non-expansive, we must have $\sqrt{2-2b}\leq\sqrt{2-2d}$, which implies that $d\leq b$. Now, consider points $(cx,d)$ and $(0,-1)$. The distance between these points is equal to $\sqrt{c^2+(1+d)^2}=\sqrt{2+2d}$. We have $F(cx,d)=(ay,b)$ and by Step 5 we have $F(0,-1)=(0,-1)$. The distance between these points is $\sqrt{a^2+(1+b)^2}=\sqrt{2+2b}$. Since $F$ is non-expansive, we must have $\sqrt{2+2b}\leq\sqrt{2+2d}$, which implies that $b\leq d$. Therefore, we have $b=d$.

Consider now points $(cx,b)$ and $(cx,0)$. The distance between these points is equal to $|b|$. We have $F(cx,b)=(ay,b)$ and by Step 2 we have $F(cx,0)=(\mathcal{L}(cx),0)$. The distance between these points is equal to $\sqrt{\|ay-\mathcal{L}(cx)\|^2+b^2}$. Since $F$ is non-expansive, we must have $\sqrt{\|ay-\mathcal{L}(cx)\|^2+b^2}\leq|b|$, but this implies $\mathcal{L}(cx)=ay$. Therefore, we have $F^{-1}(ay,b)=(\mathcal{L}^{-1}(ay),b)$, which finishes the argument.

{ \bf Step 7.} The relation $F(ax,b)=(\mathcal{L}(ax),b)$ is now established for $\|x\|=1$. Note that the latter relation is obvious for $x=0$, being a corollary of Step 5. Therefore, let us consider the case where $\|x\|\in(0,1)$.

To begin with, let us consider points $(a\frac{x}{\|x\|},b)$, $(ax,b)$, and $(0,b)$. The distance between the first two is $a(1-\|x\|)$ and the distance between the second two is $a\|x\|$. Note that $F(0,b)=(0,b)$ by Step 5 and $F(a\frac{x}{\|x\|},b)=(\mathcal{L}(a\frac{x}{\|x\|}),b)$ by Step 6. Since $F$ is non-expansive, we have \[\left\|\left(\mathcal{L}\left(a\tfrac{x}{\|x\|}\right),b\right)-F(ax,b)\right\|\leq a(1-\|x\|)\]
and
\[\left\|F(ax,b)-(0,b)\right\|\leq a\|x\|.\]
However, as the distance between $(\mathcal{L}(a\frac{x}{\|x\|}),b)$ and $(0,b)$ is equal to $a$, the last two inequalities must be in fact equalities, since otherwise we have a contradiction with the triangle inequality. Let $(\sum_{n=1}^{\infty}y_ne_n,d)$ be the expansion of $F(ax,b)$ and let $(\sum_{n=1}^{\infty}\tilde{y}_ne_n,b)$ be the expansion of $(\mathcal{L}(a\frac{x}{\|x\|}),b)$. Now we have
\begin{align*}
    a=
    \left\|\mathcal{L}\left(a\tfrac{x}{\|x\|}\right)\right\|=
    \sum_{n=1}^{\infty}|\tilde{y}_n|\leq
    \sum_{n=1}^{\infty}|\tilde{y}_n-y_n|+\sum_{n=1}^{\infty}|y_n|\leq\\
    \sqrt{\left(\sum_{n=1}^{\infty}|\tilde{y}_n-y_n|\right)^2+(b-d)^2}+\sqrt{\left(\sum_{n=1}^{\infty}|y_n|\right)^2+(b-d)^2}=\\
    \left\|\left(\mathcal{L}\left(a\tfrac{x}{\|x\|}\right),b\right)-F(ax,b)\right\|+\left\|F(ax,b)-(0,b)\right\|=\\
    a(1-\|x\|)+a\|x\|=a.
\end{align*}
Therefore, all the inequalities in between are in fact equalities, which is only possible when $b=d$.

Finally, consider points $(ax,b)$ and $(ax,0)$. The distance  between these two points is $|b|$. By Step 2 we have $F(ax,0)=(\mathcal{L}(ax),0)$. Since $F$ is non-expansive, it follows that $\sqrt{\|\sum_{n=1}^{\infty}y_ne_n-\mathcal{L}(ax)\|^2+b^2}\leq|b|$, hence $\sum_{n=1}^{\infty}y_ne_n=\mathcal{L}(ax)$. Therefore, $F(ax,b)=(\mathcal{L}(ax),b)$ as needed.

We have shown that $F$ is a restriction of an isometric automorphism of $Z$. Therefore, $F$ is an isometry as required.
\end{proof}

%% file: 4_nikita.tex
\section{The space $X\oplus_\infty Y$}

In this section, we are going to prove the plasticity of the unit ball of the \mbox{$\ell_\infty$-sum} of two strictly convex Banach spaces. 

\begin{theorem}
Let $X$ and $Y$ be two strictly convex Banach spaces. Then the unit ball of $X\oplus_\infty Y$ is a plastic metric space.
\end{theorem}

\begin{proof}
Denote the space $X\oplus_\infty Y$ by $Z$. The unit ball of this space is the set of all pairs $(x,y)$ where $x\in B_X$ and $y\in B_Y$. Extreme points are the pairs $(x,y)$ where $x\in S_X$ and $y\in S_Y$. If one of the two spaces is trivial, then $Z$ is itself strictly convex, so we can limit ourselves with the case where $X$ and $Y$ are both non-trivial.

Let $F\colon B_Z\to B_Z$ be an arbitrary non-expansive bijection from the unit ball of $Z$ onto itself. Our goal is to show that $F$ is an isometry. Let us fix the notation. Denote by $G$ the inverse of $F$. For $z\in Z$ denote by $z_x$ and $z_y$ the first and the second element of $z$. Denote by $Z_X$ the set
\[\{z\in B_Z\colon \|z_x\|>\|z_y\|\},\]
denote by $Z_Y$ the set
\[\{z\in B_Z\colon \|z_x\|<\|z_y\|\},\]
and denote by $E$ the set
\[\{z\in B_Z\colon \|z_x\|=\|z_y\|\}.\]
These three sets form a partition of the unit ball of $Z$ -- they are pairwise disjoint and their union is $B_Z$. The set $E$ is closed, the closure of $Z_X$ is $Z_X\cup E$ and the closure of $Z_Y$ is $Z_Y\cup E$. Since $E=\{\alpha z\colon \alpha\in[0,1],\ z\in\ext B_Z\}$, the last item of Proposition \ref{proposition_BnEproperties} implies $G(E)\subset E$. The proof is going to depend on the number of dimensions of $X$ and $Y$. If $X$ has more than one dimension, then $Z_X$ is a connected set. However, if $X$ is $\R$ (that is, $X$ is one-dimensional), then the set $Z_X$ has two connected components: $Z_X^-=\{z\in Z_X\colon z_x<0\}$ and $Z_X^+=\{z\in Z_X\colon z_x>0\}$. Obviously, the same can be said about $Y$ and $Z_Y$. If $X$ and $Y$ are both one-dimensional, then $Z$ is finite-dimensional, so we can omit this case. Hence, we have to handle two cases: 1) both spaces have more than one dimension, 2) one of the spaces is $\R$ and the other has more than one dimension. 

1) Let us start with the first case. We know that $G(E)\subset E$, so it follows that $F(Z_X\cup Z_Y)\subset Z_X\cup Z_Y$. The set $Z_X\cup Z_Y$ has two connected components, these are $Z_X$ and $Z_Y$, so the continuity of $F$ together with the last inclusion imply that there are four possible cases:
\begin{enumerate}
    \item[i)] $F(Z_X)\subset Z_X$ and $F(Z_Y)\subset Z_X$,
    \item[ii)] $F(Z_X)\subset Z_X$ and $F(Z_Y)\subset Z_Y$,
    \item[iii)] $F(Z_X)\subset Z_Y$ and $F(Z_Y)\subset Z_X$,
    \item[iv)] $F(Z_X)\subset Z_Y$ and $F(Z_Y)\subset Z_Y$.
\end{enumerate}
Consider case i). For this case, the continuity of $F$ implies that $F(\Cl{Z_X})\subset \Cl{Z_X}$ and $F(\Cl{Z_Y})\subset \Cl{Z_X}$, and since the union of $\Cl{Z_X}$ and $\Cl{Z_Y}$ is the unit ball, we have $F(B_Z)\subset \Cl{Z_X}$, which contradicts the surjectivity of $F$. Case iv) is similar. Therefore, we are left with cases ii) and iii), which we are going to refer to as cases A) and B).

2) Now, let us consider the case where one of the spaces is $\R$ and the other has more than one dimension. We are going to show that this reduces to case A). Suppose that $X$ has more than one dimension and $Y$ is $\R$. Let us introduce the following notations:
\begin{align*}
    Z_Y^-&=\{z\in Z_Y\colon z_y<0\},\\
    Z_Y^+&=\{z\in Z_Y\colon z_y>0\},\\
    E^-&=\{z\in E\colon z_y\leq 0\},\\
    E^+&=\{z\in E\colon z_y\geq 0\}.
\end{align*}
Note that the closure of $Z_Y^-$ is $Z_Y^-\cup E^-$ and the closure of $Z_Y^+$ is $Z_Y^+\cup E^+$. As with the previous case, the starting point is the inclusion $F(Z_X\cup Z_Y)\subset Z_X\cup Z_Y$, but the set $Z_X\cup Z_Y$ is now comprised of three connected components: $Z_Y^-$, $Z_Y^+$, and $Z_X$. The continuity of $F$ together with the last inclusion imply that there are $27=3^3$ possible cases, but some of these can be excluded, because they contradict the surjectivity of $F$. As a result, we are left with a total of $6=3!$ possible cases. We can divide these cases into three groups of two, obtaining the following three cases:
\begin{enumerate}
    \item[i)] $F(Z_X)\subset Z_X$ and $F(Z_Y)\subset Z_Y$,
    \item[ii)] $F(Z_X)\subset Z_Y^-$ and $F(Z_Y)\subset Z_X\cup Z_Y^+$,
    \item[iii)] $F(Z_X)\subset Z_Y^+$ and $F(Z_Y)\subset Z_X\cup Z_Y^-$.
\end{enumerate}
The first case is the same as case A). Our goal is to exclude cases ii) and iii). Consider case ii). The continuity of $F$ implies the inclusions $F(\Cl{Z_X})\subset Z_Y^-\cup E^-$ and $F(\Cl{Z_Y})\subset Z_X\cup Z_Y^+\cup E$. Since $E$ is a subset of $\Cl{Z_X}$ and a subset of $\Cl{Z_Y}$, we obtain the inclusion $F(E)\subset E^-$, which contradicts the previously obtained inclusion $G(E)\subset E$. Case iii) is similar. As we see, we are left with case A).

Thereby, we have two cases to consider: A) $F(Z_X)\subset Z_X$ and $F(Z_Y)\subset Z_Y$, B) $F(Z_X)\subset Z_Y$ and $F(Z_Y)\subset Z_X$. Let us make some additional conclusions. For case A), the continuity of $F$ implies that $F(\Cl{Z_X})\subset \Cl{Z_X}$ and $F(\Cl{Z_Y})\subset \Cl{Z_Y}$. As $E$ is a subset of $\Cl{Z_X}$ and a subset of $\Cl{Z_Y}$, it follows that $F(E)\subset \Cl{Z_X}\cap \Cl{Z_Y}$. The latter is equal to $E$, so we have the inclusion $F(E)\subset E$. We can use the same argument to show this for case B). Now, for case A) we have $F(Z_X)\subset Z_X$, $F(Z_Y)\subset Z_Y$ and $F(E)\subset E$. Since $F$ is a bijection, while $Z_X$, $Z_Y$, and $E$ form a partition of $B_Z$, we can conclude that $F(Z_X)=Z_X$, $F(Z_Y)=Z_Y$, and $F(E)=E$. Analogously, for case B) we obtain $F(Z_X)=Z_Y$, $F(Z_Y)=Z_X$, and $F(E)=E$.

The last item of Proposition \ref{proposition_BnEproperties} asserts that $G(\ext B_Z)\subset \ext B_Z$. We can also show that $F(\ext B_Z)\subset\ext B_Z$. To demonstrate this, let $z\in\ext B_Z$ be arbitrary. Since $F(E)=E$ and $z$ belongs to $E$, $F(z)$ belongs to $E$ as well, so $F(z)=\alpha z'$ where $z'\in\ext B_Z$ and $\alpha\in[0,1]$. We need to show that $\alpha=1$. The last item of Proposition \ref{proposition_BnEproperties} implies that $G(\alpha z')=\alpha G(z')$ and $G(z')\in\ext B_Z$, so we obtain the equation $z=\alpha G(z')$. Taking the norm of both sides, we see that $\alpha=1$. This proves the inclusion $F(\ext B_Z)\subset\ext B_Z$, hence $F(\ext B_Z)= \ext B_Z$.

To finish the proof, we would have to consider cases A) and B) separately. However, the proofs for these two cases are almost identical, so we will only consider case A).

First, we are going to show that
\[\forall z,w\in B_Z\quad (z_x=w_x\in S_X \implies G(z)_x=G(w)_x\in S_X),\]
\[\forall z,w\in B_Z\quad (z_y=w_y\in S_Y \implies G(z)_y=G(w)_y\in S_Y).\]
Let us demonstrate the first item. Fix $z,w\in B_Z$ and suppose $z_x=w_x\in S_X$. Let us consider four possible cases.
\begin{itemize}
\item First, we consider the case where $z_y\in S_Y$ and $w_y\not\in S_Y$. Since $w_x\in S_X$ and $w_y\not\in S_Y$, it follows that $w\in Z_X$. As $w\in Z_X$ and $F(Z_X)=Z_X$, we have $G(w)\in Z_X$. Therefore, $G(w)_y\not\in S_Y$. Since $z_x=w_x\in S_X$, the distance between $-z$ and $w$ is two. Therefore, the distance between $G(-z)$ and $G(w)$ is also two. This means that either $G(-z)_x$ and $G(w)_x$ are distance two apart or $G(-z)_y$ and $G(w)_y$ are distance two apart. The second possibility is excluded, because $G(w)_y\not\in S_Y$. Consequently, we have the first case. As $G(-z)_x$ and $G(w)_x$ are distance two apart, they should belong to $S_X$. Moreover, since $X$ is strictly convex, it follows that $G(-z)_x=-G(w)_x$. Finally, since $z\in\ext B_Z$, the last item of Proposition \ref{proposition_BnEproperties} implies that $G(-z)=-G(z)$, so we have $G(z)_x=G(w)_x\in S_X$ as wanted.
\item The same argument can be applied to the case where $z_y\not\in S_Y$ and $w_y\in S_Y$ by swapping the roles of $z$ and $w$.
\item If $z_y\in S_Y$ and $w_y\in S_Y$, then we take $p\in B_Z$ such that $p_y\not\in S_Y$ and $p_x$ is the same as $z_x$ and $w_x$. Then $z_x=p_x\in S_X$, $z_y\in S_Y$, and $p_y\not\in S_Y$, so $G(z)_x=G(p)_x\in S_X$ by the first item of this list. Similarly, as $w_x=p_x\in S_X$, $w_y\in S_Y$, and $p_y\not\in S_Y$, we have $G(w)_x=G(p)_x\in S_X$. Combining these two yields $G(z)_x=G(w)_x\in S_X$ as wanted.
\item An argument similar to the one used in the third item of this list can be applied to the case where $z_y\not\in S_Y$ and $w_y\not\in S_Y$ by choosing $p\in B_Z$ such that $p_y\in S_Y$ and $p_x$ is the same as $z_x$ and $w_x$.
\end{itemize}
We have considered all four possible cases. One can use the same argument to prove the second item. The obtained result brings us to consider functions $g_x\colon S_X\rightarrow S_X$ and $g_y:S_Y\rightarrow S_Y$ that can be defined with $g_x(x)=G(x,0)_x$ and $g_y(y)=G(0,y)_y$. So now we have
\[\forall z\in B_Z\quad (z_x\in S_X \implies G(z)_x=g_x(z_x)),\]
\[\forall z\in B_Z\quad (z_y\in S_Y \implies G(z)_y=g_y(z_y)).\]

We go on to prove some facts about these functions. First, let us show that these functions are injective. Consider $g_x$. Suppose by contrary that there are $x,x'\in S_X$ such that $x\ne x'$ and $g_x(x)=g_x(x')$. Choose arbitrary $y\in S_Y$ and consider pairs $(x,y)$ and $(x',y)$. We have $G(x,y)=(g_x(x),g_y(y))$ and $G(x',y)=(g_x(x'),g_y(y))$. Note that $(x,y)\ne(x',y)$, but $G(x,y)=G(x',y)$. This contradicts the injectivity of $G$. Thus, $g_x$ is actually injective. We can use the same argument to show that $g_y$ is injective too. 

Let us consider surjectivity. Let $x\in S_X$ and $y\in S_Y$ be arbitrary. Then $(x,y)$ is extreme and hence $F(x,y)$ is also extreme, because $F(\ext B_Z)=\ext B_Z$. Denote by $x'$ and $y'$ the first and the second element of $F(x,y)$. Since $(x',y')$ is extreme, we have $x'\in S_X$ and $y'\in S_Y$. Since $F$ sends $(x,y)$ to $(x',y')$, we have $G(x',y') = (x,y)$, thus $g_x(x')=x$ and $g_y(y')=y$, which proves the surjectivity of both $g_x$ and $g_y$. Now that we know $g_x$ and $g_y$ are bijective, denote the inverse of $g_x$ by $f_x$ and the inverse of $g_y$ by $f_y$.

Let us show that $g_x$ and $g_y$ are symmetric. That is, we demonstrate that $g_x(-x)=-g_x(x)$ and $g_y(-y)=-g_y(y)$ for every $x\in S_X$ and $y\in S_Y$. Pick an arbitrary $x\in S_X$. Let us show that $g_x(-x)=-g_x(x)$. Choose an arbitrary $y\in S_Y$. Consider pairs $z=(x,y)$ and $w=(-x,y)$. The distance between them is two, so the distance between $G(z)$ and $G(w)$ should be also two. We know that $G(z)=(g_x(x),g_y(y))$ and $G(w)=(g_x(-x),g_y(y))$. For distance between these two pairs to be equal to two, elements $g_x(x)$ and $g_x(-x)$ should be distance two apart. Since $X$ is a strictly convex space, it follows that $g_x(-x)=-g_x(x)$. The same argument can be used to show that $g_y(-y)=-g_y(y)$ for $y\in S_Y$. Obviously, the symmetricity of $g_x$ and $g_y$ implies that $f_x$ and $f_y$ are also symmetric.

Define a function $G_x:X\rightarrow X$ with $G_x(0)=0$ and $G_x(tx)=tg_x(x)$ for all $x\in S_X$ and $t>0$. Similarly, define a function $G_y:Y\rightarrow Y$ with $G_y(0)=0$ and $G_y(ty)=tg_y(y)$ for all $y\in S_Y$ and $t>0$. We see that $G_x$ coincides with $g_x$ on $S_X$ and $G_y$ coincides with $g_y$ on $S_Y$. Using the properties of $g_x$ and $g_y$ established above, one can show that $G_x$ and $G_y$ are bijective and homogeneous (that is, for every $x\in B_X$, $y\in B_Y$ and $t\in\mathbb{R}$ we have $G_x(tx)=tG_x(x)$ and $G_y(ty)=tG_y(y)$). Denote the inverses of $G_x$ and $G_y$ by $F_x$ and $F_y$. It is easy to see that $F_x$ and $F_y$ are also homogeneous. Moreover, we see that $F_x$ coincides with $f_x$ on $S_X$ and $F_y$ coincides with $f_y$ on $S_Y$.

Now, we are going to examine what $G$ does with a pair one element of which lies on the sphere and the other element of which is arbitrary. Our goal is to show that
\[\forall x\in S_X\ \forall y\in B_Y\quad G(x,y)=(G_x(x),G_y(y)),\]
\[\forall x\in B_X\ \forall y\in S_Y\quad G(x,y)=(G_x(x),G_y(y)).\]
Let us consider the first item. Fix arbitrary $x\in S_X$ and arbitrary $y\in B_Y$. First, let us consider the case $y=0$. Our goal is to show that $G(x,0)=(g_x(x),0)$. We know that $G(x,0)=(g_x(x),y_0)$ where $y_0\in B_Y$. Suppose by contrary that $y_0\ne 0$. Denote by $y_0^*$ the element $y_0/\|y_0\|\in S_Y$. We see that $G$ sends $(x,0)$ to $(g_x(x),y_0)$ and $(x,f_y(y_0^*))$ to $(g_x(x),y_0^*)$. The distance between $(x,0)$ and $(x,f_y(y_0^*))$ is one, while the distance between $(g_x(x),y_0)$ and $(g_x(x),y_0^*)$ is smaller than one, which is a contradiction. Next, consider the case $y\ne 0$. Then $y=tu$ where $u\in S_Y$ and $t\in(0,1]$. Our goal is to show that $G$ sends $(x,tu)$ to $(g_x(x),tg_y(u))$. Again, we know that $G(x,tu)=(g_x(x),y_0)$ where $y_0\in B_Y$. We need to show that $y_0=tg_y(u)$. If $y_0=0$, then $G$ sends both $(x,0)$ and $(x,tu)$ to $(g_x(x),0)$, which contradicts the injectivity of $G$. Therefore, $y_0\ne 0$ and we can consider the element $y_0^*=y_0/\|y_0\|\in S_Y$. We see that $G$ sends $(x,f_y(y_0^*))$ to $(g_x(x),y_0^*)$, $(x,tu)$ to $(g_x(x),y_0)$, and $(x,0)$ to $(g_x(x),0)$. The distance between $(g_x(x),y_0^*)$ and $(g_x(x),y_0)$ is $1-\|y_0\|$, the distance between $(g_x(x),y_0)$ and $(g_x(x),0)$ is $\|y_0\|$. So the distance between $(x,f_y(y_0^*))$ and $(x,tu)$ should be at most $1-\|y_0\|$ and the distance between $(x,tu)$ and $(x,0)$ should be at most $\|y_0\|$. This means that the distance between $f_y(y_0^*)$ and $tu$ should also be at most $1-\|y_0\|$ and the distance between $tu$ and $0$ should also be at most $\|y_0\|$. Since $Y$ is a strictly convex space, we have $tu=\|y_0\|f_y(y_0^*)$. It follows that $t=\|y_0\|$ and $u=f_y(y_0^*)$. Thus we have $y_0=\|y_0\|y_0^*=tg_y(u)$ as wanted. The proof of the second item is analogous.

From the last obtained result it follows directly that $F$ has the analogous properties. That is, we have
\[\forall x\in S_X\ \forall y\in B_Y\quad F(x,y)=(F_x(x),F_y(y)),\]
\[\forall x\in B_X\ \forall y\in S_Y\quad F(x,y)=(F_x(x),F_y(y)).\]

Finally, we show that for every $z\in B_Z$ we have $F(z)=(F_x(z_x),F_y(z_y))$. Let us first show that for every $z\in B_Z$ we have $F(z)_x=F_x(z_x)$. Fix arbitrary $z\in B_Z$. Denote the elements of $z$ by $x$ and $y$ and the elements of $F(z)$ by $x'$ and $y'$. The element $x\in B_X$ can be represented in the form $tu$ where $t\in[0,1]$ and $u\in S_X$. As $F_x(tu)=tF_x(u)$ by homogeneity of $F_x$, the equality that we need to prove takes the form $x'=tF_x(u)$. We know that $F$ sends $(tu,y)$ to $(x',y')$, $(u,y)$ to $(F_x(u),F_y(y))$ and $(-u,y)$ to $(F_x(-u),F_y(y))$. By homogeneity of $F_x$ we have $F_x(-u)=-F_x(u)$. The distance between $(u,y)$ and $(tu,y)$ is $1-t$, the distance between $(tu,y)$ and $(-u,y)$ is $1+t$. Thus the distance between $(F_x(u),F_y(y))$ and $(x',y')$ is at most $1-t$ and the distance between $(x',y')$ and $(-F_x(u),F_y(y))$ is at most $1+t$. This implies that the distance between $F_x(u)$ and $x'$ is also at most $1-t$ and the distance between $x'$ and $-F_x(u)$ is also at most $1+t$. Since $X$ is a strictly convex space, it follows that $x'=tF_x(u)$ as wanted. One can use the same argument to show that for every $z\in B_Z$ we have $F(z)_y=F_y(z_y)$.

We can finish the proof by applying Theorem \ref{theorem_sufficient}.

\end{proof}